\documentclass[10pt]{article}
\usepackage{amsfonts}
\usepackage{amsmath}
\usepackage{mathrsfs}
\usepackage{color}
\usepackage{tikz}
\usepackage{mathrsfs,amscd,amssymb,amsthm,amsmath,bm,graphicx,psfrag,subfigure,url}

\usepackage{indentfirst}

\def \[{\begin{equation}}
\def \]{\end{equation}}

\newtheorem{thm}{Theorem}[section]

\newtheorem{Case}{Case}

\newtheorem{lem}[thm]{Lemma}

\newtheorem{pb}[thm]{Problem}

\newenvironment{wst}
{\setlength{\leftmargini}{1.5\parindent}
 \begin{itemize}
 \setlength{\itemsep}{-1.1mm}}
{\end{itemize}}

\begin{document}

\title{\bf On the largest and least eigenvalues of eccentricity matrix of trees }

\author{Xiaocong He\footnote{Corresponding author}}
\date{}

\maketitle

\begin{center}
School of Mathematics and Statistics, Central South University, New Campus, \\[2pt]
Changsha, Hunan, 410083, PR China\\[5pt]
hexc2018@qq.com (X.C.~He)
\medskip
\end{center}

\begin{abstract}
The eccentricity matrix $\varepsilon(G)$ of a graph $G$ is constructed from the distance matrix of $G$ by keeping only the largest distances for each row and each column. This matrix can be interpreted as the opposite of the adjacency matrix obtained from the distance matrix by keeping only the distances equal to 1 for each row and each column. The $\varepsilon$-eigenvalues of a graph $G$ are those of its eccentricity matrix $\varepsilon(G)$. Wang et al \cite{e} proposed the problem of determining the maximum $\varepsilon$-spectral radius of trees with given order. In this paper, we consider the above problem of $n$-vertex trees with given diameter. The maximum $\varepsilon$-spectral radius of $n$-vertex trees with fixed odd diameter is obtained, and the corresponding extremal trees are also determined. Recently, Wei et al. \cite{w} determined all connected graphs on $n$ vertices of maximum degree less than $n-1$, whose least eccentricity eigenvalues are in $[-2\sqrt{2}, -2]$. Denote by $S_{n}$ the star on $n$ vertices. For tree $T$ with order $n\geq3$, it \cite{w} was proved that $\varepsilon_n(T)\leq-2$ with equality if and only if $T\cong S_n$. According to the above results, the trees of order $n\geq3$ with least $\varepsilon$-eigenvalues in $[-2\sqrt{2},0)$ are only $S_n$. Motivated by \cite{w}, we determine the trees with least $\varepsilon$-eigenvalues in $[-2-\sqrt{13},-2\sqrt{2})$.
\end{abstract}

\vspace{2mm} \noindent{\bf Keywords}: The eccentricity matrix; Spectral radius; The least eigenvalue; Diameter
\vspace{2mm}

\setcounter{section}{0}
\section{Introduction}\setcounter{equation}{0}
In this paper, we only consider connected and simple graphs, and refer to Bondy and Murty \cite{2} for notations and terminologies used but not defined here.

Let $G$ be a graph with vertex set $V(G)$ and edge set $E(G)$. $G-v$ (resp. $G-uv$) is the graph obtained from $G$ by deleting vertex $v$ together with incident edges (resp. edge $uv\in E(G)$). This notation is naturally extended if more than one vertex or edge are deleted. Similarly, $G+uv$ is obtained from $G$ by adding an edge $uv \notin E(G)$. If $U\subseteq V(G)$, then we write $G[U]$ to denote the induced subgraph of $G$ with vertex set $U$ and two vertices being adjacent if and only if they are adjacent in $G$. A \textit{pendant vertex} is the vertex of degree 1 and a \textit{supporting vertex} is the neighbor of a pendant vertex. A \textit{pendant edge} is an edge which is incident to a supporting vertex and a pendant vertex. Denote by $P_n, C_n, S_{n}$ and $K_n$ the path, cycle, star and complete graph on $n$ vertices, respectively. An \textit{acyclic} graph is one that contains no cycles. A connected acyclic graph is called a tree. If the tree is nontrivial, a vertex of degree one is called a leaf of the tree. A \textit{caterpillar tree} is a tree with a single path containing at least one endpoint of every edge. For a real number $x$, denote by $\lfloor x \rfloor$ the greatest integer no more than $x$, and by $\lceil x \rceil$ the least integer no less than $x$. We may denote the $n\times n$ identity matrix by ${\bf I}_n$.

We denote the \textit{neighbors} of vertex $u$ and the \textit{degree} of vertex $u$ in graph $G$ by $N_G(u)$ and $d_{G}(u)$, respectively. The \textit{distance} $d_{G}(u,v)$ between vertices $u$ and $v$ is the length of a shortest path between them in $G$ and the \textit{eccentricity} of vertex $u$ is defined as $e_{G}(u)=\max\{d_{G}(u, v)|v\in V(G)\}$. Then the \textit{diameter} of $G$, written as $diam(G)$, is
$
\max\{e_G(u)|u\in V(G)\}.
$
A \textit{diametrical path} is a path whose length is equal to the diameter of $G$.

Let $M(G)$ be an $n\times n$ matrix closely related to the structural theory of a graph $G$. Then the \textit{$M$-polynomial} of $G$ is defined as $\varphi_{M}(G,\lambda)=\det(\lambda I_n-M(G))$, and the roots of $\varphi_{M}(G,\lambda)=0$ are the \textit{$M$-eigenvalues}. The \textit{$M$-spectrum} $Spec_M(G)$ of $G$ is a multiset consisting of the distinct $M$-eigenvalues together with their multiplicities, in which the maximum modulus is called the \textit{$M$-spectral radius} of $G$. It is well-known that there are several classical graph matrices, including adjacency matrix, distance matrix, Laplacian matrix, signless Laplacian matrix, resistance matrix and so on.

Let $D(G)$ be the \textit{distance matrix} of $G$ with $(u,v)$-entry $(D(G))_{uv}=d_{G}(u,v)$. The \textit{eccentricity matrix} $\varepsilon(G)$ of $G$ is constructed from the distance matrix $D(G)$ by only retaining the eccentricities in each row and each column and setting the rest elements in the corresponding row and column to be zero. To be more precise, the $(u,v)$-entry of eccentricity matrix is defined as
$$
(\varepsilon(G))_{uv}=\left\{
                      \begin{array}{ll}
                        (D(G))_{uv}, & \hbox{if $(D(G))_{uv}=\min\{e_G(u),e_G(v)\}$;} \\[5pt]
                        0, & \hbox{otherwise.}
                      \end{array}
                      \right.
$$
It is obvious that $\varepsilon(G)$ is real and symmetric. Then the $\varepsilon$-eigenvalues of $G$ are real, denoted by $\varepsilon_1(G)\geqslant\varepsilon_2(G)\geqslant\cdots\geqslant\varepsilon_n(G)$. Randi\'{c} et al. \cite{9,10} defined so-called $D_{MAX}$ \textit{matrix}, which was renamed as the eccentricity matrix by Wang et al. \cite{11}. Furthermore, Dehmer and Shi \cite{6} studied the uniqueness of $D_{MAX}$-matrix. Recently, Wang et al. \cite{en} studied the graph energy based on the eccentricity matrix; Wang et al. \cite{e} studied some spectral properties of the eccentricity matrix of graphs; Mahato et al. \cite{s} studied the spectra of graphs based on the eccentricity matrix; Wei et al. \cite{w} determined the $n$-vertex trees with minimum $\varepsilon$-spectral radius. Furthermore, in \cite{w}, the authors identified all trees with given order and diameter having minimum $\varepsilon$-spectral; Tura et al. \cite{ar} studied the eccentricity energy of complete multipartite graphs.

Note that the adjacency matrix $A(G)$ can be regarded as constructed from the distance matrix $D(G)$ by selecting only the smallest distances for each row and each column, which correspond to adjacent vertices. From this point of view, the eccentricity matrix can be viewed as the opposite to the adjacency matrix \cite{11} and these two matrices express two extremes of distance-like matrix.

The adjacency and distance matrices have been extensively studied and applied; see \cite{1,3,4,5,ff,fc,fz,fm,hq,fd,zs,lm,zm,z1,z2}. One of the most important facts is that the adjacency and distance matrices of connected graphs are irreducible, but it does not hold for all eccentricity matrices. Let $T$ be a tree with at least two vertices. Recently, Wang et al. \cite{11} proved that the eccentricity matrix of $T$ is irreducible and they characterized the relationships between the $A$-eigenvalues and $\varepsilon$-eigenvalues of some graphs. Then $\varepsilon$-spectral radius $\varepsilon_1(T)$ is positive and there is an eigenvector corresponding to $\varepsilon_1(T)$, called \textit{Perron eigenvector}, whose each coordinate is positive by Perron-Frobenius Theorem. Let $M$ and $N$ be two matrices with same order. If $(N)_{ij}\leqslant (M)_{ij}$ for each $i,j$, we let $N\leqslant M$.

In view of more novel properties of eccentricity matrix, further discussion is needed. In particular, Wang et al. \cite{e} proposed the following problem.

\begin{pb}[\cite{e}]\label{conj1.1}
Which trees have the maximum $\varepsilon$-spectral radius?
\end{pb}

Recently, Wei et al. \cite{w} determined all connected graphs on $n$ vertices of maximum degree less than $n-1$, whose least eccentricity eigenvalues are in $[-2\sqrt{2}, -2]$. For tree $T$ with order $n\geq3$, it \cite{w} was proved that $\varepsilon_n(T)\leq-2$ with equality if and only if $T\cong S_n$.

Motivated by the above results, we now propose the following problem.
\begin{pb}
For some given number $c<-2\sqrt{2}$, which trees with least eccentricity eigenvalues are in $[c, -2\sqrt{2})$?
\end{pb}

In this paper, we characterize the extremal trees having maximum $\varepsilon$-spectral radius with given order and odd diameter. On the other hand, we determine all the trees  with least eccentricity eigenvalues in $[-2-\sqrt{13}, -2\sqrt{2})$.

Further on we need the following lemmas.
\begin{lem}[\cite{8}]\label{lem2.1}
Let $M$ be a Hermitian matrix of order $s$, and let $N$ be a principle submatrix of $M$ with order $t$. If $\lambda_1\geqslant \lambda_2\geqslant \cdots \geqslant \lambda_s$ list the eigenvalues of $M$ and $\mu_1\geqslant \mu_2\geqslant \cdots \geqslant \mu_t$ are the eigenvalues of $N$, then $\lambda_{i}\geqslant \mu_i\geqslant \lambda_{s-t+i}$ for $1\leqslant i\leqslant t$.
\end{lem}
\begin{lem}[\cite{8}]\label{lem2.2}
Let $M$ and $N$ be two nonnegative irreducible matrices with same order. If $(N)_{ij}\leqslant (M)_{ij}$ for each $i,j$, then $\rho(N)\leqslant \rho(M)$ with equality if and only if $M=N$, where $\rho(N)$ and $\rho(M)$ denote the \textit{spectral radius} of $N$ and $M$, respectively.
\end{lem}
\begin{lem}[\cite{11}]\label{lem2.3}
The eccentricity matrix $\varepsilon(T)$ of a tree $T$ with at least two vertices is irreducible.
\end{lem}
\begin{lem}[\cite{w}]\label{lem2.4}
Let $G$ be an $n$-vertex connected graph with diameter $d$. Then $\varepsilon_1(G)\geqslant d$ and $\varepsilon_n(G)\leqslant -d$.
\end{lem}
\begin{lem}[\cite{w}]\label{lem2.5}
Let $T$ be a tree with order $n\geq3$. Then $\varepsilon_n(T)\leqslant -2$ with equality if and only if $T\cong S_n$.
\end{lem}
\section{The maximum $\varepsilon$-spectral radius of trees with fixed odd diameter}

In this section, we characterize the extremal trees with fixed odd diameter having maximum $\varepsilon$-spectral radius. Firstly, we present a few technical lemmas aiming to provide some fundamental characterizations of extremal trees.

Denote by $\mathscr{T}_{n,d}$ the set of trees with order $n$ and diameter $d$. It is easy to check that the tree with diameter 1 is $K_2$ and the tree with diameter 2 is a star with at least 3 vertices.

If $d\geqslant3$ is odd, let $D_{n,d}^{a,b}$ be the tree obtained from $P_{d+1}=v_0v_1v_2 \cdots v_d$ by attaching $a$ pendant vertices to $v_1$ and $b$ pendant vertices to $v_{d-1}$, where $a+b=n-d-1$ and $b\geqslant a\geqslant 0$, as depicted in Fig. 1.

\begin{figure}[!ht]
\begin{center}
\psfrag{a}{$v_0$}
\psfrag{b}{$v_1$}
\psfrag{c}{$v_{2}$}
\psfrag{f}{$v_{d-2}$}
\psfrag{g}{$v_{d-1}$}
\psfrag{h}{$v_d$}
\psfrag{i}{$u_1$}
\psfrag{j}{$u_a$}
\psfrag{k}{$w_1$}
\psfrag{l}{$w_b$}
\psfrag{2}{$\cdots$}
\psfrag{1}{$D_{n,d}^{a,b}$}
\includegraphics[width=90mm]{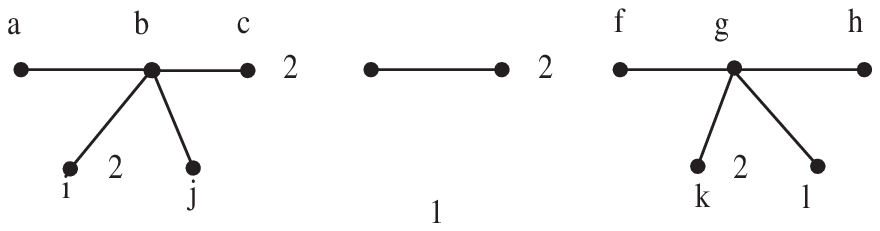} \\
\caption{Tree $D_{n,d}^{a,b}$.}
\end{center}
\end{figure}

\begin{lem}\label{lem3.1}
Let $D_{n,3}^{a,b}$ be in $\mathscr{T}_{n,3}$ defined above, where $a+b=n-4$ and $b\geqslant a\geqslant 1$. Then $\varepsilon_1(D_{n,3}^{a-1,b+1})< \varepsilon_1(D_{n,3}^{a,b})$.
\end{lem}
\begin{proof}
Choose a diametrical path $P=v_0v_1v_2v_3$ in $D_{n,3}^{a,b}$. Denote by $U$ the set of pendant neighbors of $v_1$ and let $W$ be the set of pendant neighbors of $v_2$ in $D_{n,3}^{a,b}$. It is obvious that $v_0\in U, v_3\in W$ and $|U|=a+1, |W|=b+1$.

Let ${\bf x}$ be a Perron eigenvector corresponding to $\rho:=\varepsilon_1(D_{n,3}^{a,b})$, whose coordinate with respect to vertex $v$ is ${\bf x}_{v}$. Since $\rho {\bf x}_u=2{\bf x}_{v_2}+3\sum_{w\in V_2}{\bf x}_w$ for each $u\in U$, we can get ${\bf x}_u={\bf x}_{u'}$ for $\{u,u'\}\subseteq U$. Similarly, ${\bf x}_w={\bf x}_{w'}$ for $\{w,w'\}\subseteq W$. Thus, we obtain
\begin{align*}
&\rho {\bf x}_{v_1}=2(b+1){\bf x}_w; \\
&\rho {\bf x}_{v_2}=2(a+1){\bf x}_u; \\
&\rho {\bf x}_u=2{\bf x}_{v_2}+3(b+1){\bf x}_w; \\
&\rho {\bf x}_w=3(a+1){\bf x}_u+2{\bf x}_{v_1}
\end{align*}
for any $u\in U,w\in W$. Then $\rho$ is the largest eigenvalue of
\begin{equation*}
\left(
  \begin{array}{cccc}
  0 & 0 & 0      & 2(b+1) \\
  0 & 0 & 2(a+1) & 0 \\
  0 & 2 & 0      & 3(b+1) \\
  2 & 0 & 3(a+1) & 0 \\
  \end{array}
  \right).
\end{equation*}
By calculation and the fact $b=n-4-a$, we have $\rho$ is the largest root of $f_a(\lambda)=0$ where
\begin{align}\label{3.1}
f_a(\lambda)&:=\left|\begin{array}{cccc}
  \lambda & 0       & 0       & -2(b+1) \\
  0       & \lambda & -2(a+1) & 0 \\
  0       & -2      & \lambda & -3(b+1) \\
  -2      & 0       & -3(a+1) & \lambda \\
               \end{array}\right| \notag \\
&=\lambda^4+(9a^2+36a-9na-13n+35)\lambda^2-16a^2-64a+16na+16n-48.
\end{align}

Let $\widehat{\rho}$ be the largest root of $f_{a-1}(\lambda)=0$, then $\varepsilon_1(D_{n,3}^{a-1,b+1})=\widehat{\rho}\geqslant diam(D_{n,3}^{a-1,b+1})=3$ by Lemma \ref{lem2.4}. Hence, we have
\begin{align*}
f_a(\widehat{\rho})=&f_a(\widehat{\rho})-f_{a-1}(\widehat{\rho})=(n-3-2a)(-9\widehat{\rho}^2+16)\\[5pt]
                    \leqslant&(n-3-2a)(-9\cdot3^2+16)=-65(n-3-2a).
\end{align*}
Note that $n-3>n-4=a+b\geqslant2a$ due to $b\geqslant a$. Therefore, $f_a(\widehat{\rho})<0$ then $\varepsilon_1(D_{n,3}^{a-1,b+1})=\widehat{\rho}<\rho= \varepsilon_1(D_{n,3}^{a,b})$.
\end{proof}

Our first main result in this section determines the unique tree among $\mathscr{T}_{n,3}$, having the maximum $\varepsilon$-spectral radius.
\begin{thm}\label{thm3.2}
The maximum $\varepsilon$-spectral radius is achieved uniquely by tree $D_{n,3}^{\lfloor \frac{n-4}{2}\rfloor, \lceil \frac{n-4}{2}\rceil}$ among all the trees in $\mathscr{T}_{n,3}$.
\end{thm}
\begin{proof}
According to Lemma \ref{lem3.1}, it is easy to see that the maximum $\varepsilon$-spectral radius is achieved uniquely by tree $D_{n,3}^{a,b}$ satisfying $|b-a|\leq1$ among all the trees in $\mathscr{T}_{n,3}$.

This completes the proof.
\end{proof}
\begin{figure}[!ht]
\begin{center}
\psfrag{a}{$v_0$}
\psfrag{b}{$v_1$}
\psfrag{c}{$v_{i-1}$}
\psfrag{d}{$v_i$}
\psfrag{e}{$v_{i+1}$}
\psfrag{f}{$v_{d-1}$}
\psfrag{g}{$v_d$}
\psfrag{h}{$u_{0}$}
\psfrag{i}{$u$}
\psfrag{j}{$u_1$}
\psfrag{7}{$u_2$}
\psfrag{k}{$u_s$}
\psfrag{1}{$T_1$}
\psfrag{2}{$T_{i-1}$}
\psfrag{3}{$T_{i}$}
\psfrag{4}{$T_{i+1}$}
\psfrag{5}{$T_{d-1}$}
\psfrag{6}{$\widetilde{T}_{i}$}
\psfrag{x}{$T$}
\psfrag{y}{$\widetilde{T}$}
\psfrag{0}{$\cdots$}
\includegraphics[width=150mm]{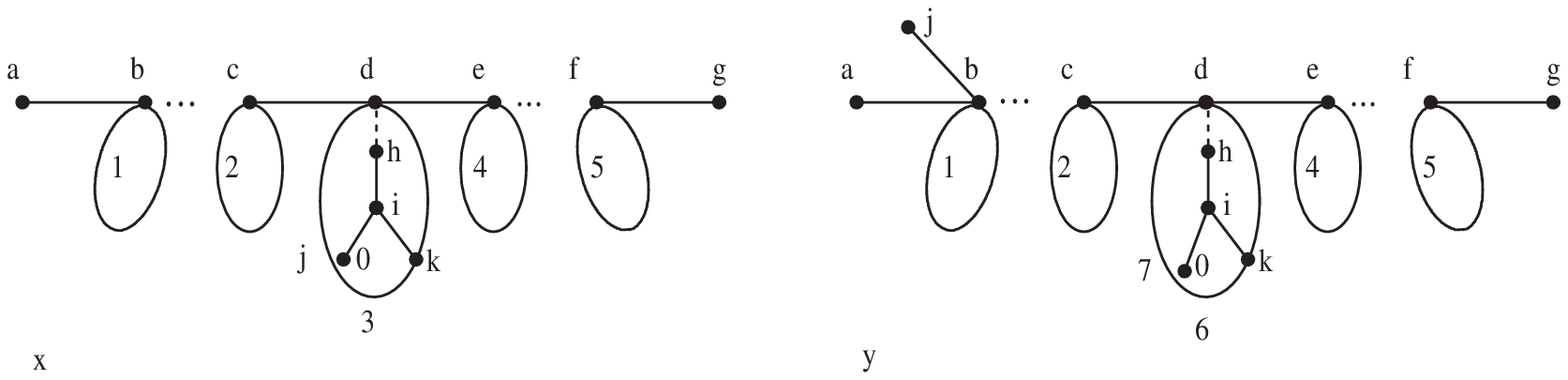} \\
\caption{Trees $T$ and $\widetilde{T}$ in Lemma \ref{lem3.3}.}
\end{center}
\end{figure}
\begin{lem}\label{lem3.3}
Let $T$ be in $\mathscr{T}_{n,d}$ with a diametrical path $P_{d+1}=v_0v_1v_2 \cdots v_d$ ($d\geqslant 5$ is odd), and let $T_j$ be the connected component of $T-E(P_{d+1})$ containing $v_j, j\in\{0,1,\ldots, d\}$. Assume there exists a vertex $u_1\in V(T_i)$ ($2\leq i\leq \frac{d-1}{2}$) such that $d_{T_i}(v_i,u_1)=e_{T_i}(v_i)\geqslant 2$ (obviously, $u_1$ is a pendant vertex). Denote the unique neighbor of $u_1$ by $u$ and all neighbors of $u$ by $u_{0},u_1,\ldots,u_s$ with $d_T(u_{0})\geqslant 2$ and $d_T(u_{j})=1$ for $1\leqslant j\leqslant s$  (see Fig. $2$). Let
$$
\widetilde{T}=T-uu_1+u_1v_1.
$$
Then $\varepsilon_1(T)\leq\varepsilon_1(\widetilde{T})$, with equality if and only if $d_{T_i}(v_i,u_1)=e_{T_i}(v_i)=i$.
\end{lem}
\begin{proof}
It is easy to check that $e_T(w)=e_{\widetilde{T}}(w)$ for each vertex $w\in V(T)\setminus\{u_1\}$, and $d_T(w,w')=d_{\widetilde{T}}(w,w')$ for $\{w,w'\}\subseteq V(T)\setminus\{u_1\}$. By the definition of eccentricity matrix, the $(w,w')$-entry of $\varepsilon(T)$ is equal to the $(w,w')$-entry of $\varepsilon(\widetilde{T})$ for each $\{w,w'\}\subseteq V(T)\setminus\{u_1\}$.

For $w\in \bigcup_{0\leq j\leq\frac{d-1}{2}}V(T_j)$, note that
\begin{align*}
&\text{$d_{T}(u_1, w)\leq d_{T}(u_1, v_{\frac{d-1}{2}})+d_{T}(w, v_{\frac{d-1}{2}})<\min\{e_T(w), e_T(u_1)\}$,}\\[5pt]
&\text{$d_{\widetilde{T}}(u_1, w)\leq d_{\widetilde{T}}(u_1, v_{\frac{d-1}{2}})+d_{\widetilde{T}}(w, v_{\frac{d-1}{2}})<\min\{e_{\widetilde{T}}(w), e_{\widetilde{T}}(u_1)\}.$}
\end{align*}
Hence, $(\varepsilon(T))_{u_1w}=0=(\varepsilon(\widetilde{T}))_{u_1w}$ for $w\in \bigcup_{0\leq j\leq\frac{d-1}{2}}V(T_j)$.

We proceed to consider the following two possible cases.
\begin{Case}
$e_{T_i}(v_i)<i$
\end{Case}
In this case, for $w\in \bigcup_{\frac{d+1}{2} \leq j\leq d}V(T_j)$, we have
\begin{align*}
&\text{$e_T(u_1)=d_{T}(u_1, v_d)\geq d_{T}(u_1, w)$,}\\[5pt]
&\text{$e_T(w)=d_{T}(v_0, w)>d_{T}(u_1, w).$}
\end{align*}
By the definition of eccentricity matrix, if $d_{T}(u_1, v_d)>d_{T}(u_1, w)$, then $(\varepsilon(T))_{u_1w}=0$. If $d_{T}(u_1, v_d)=d_{T}(u_1, w)$, then $(\varepsilon(T))_{u_1w}=d_{T}(u_1, w)$. Hence, we have $(\varepsilon(T))_{u_1w}\leq d_{T}(u_1, w)$.

On the other hand, for $w\in \bigcup_{\frac{d+1}{2} \leq j\leq d}V(T_j)$, we have
\begin{align*}
&\text{$e_{\widetilde{T}}(u_1)=d_{\widetilde{T}}(u_1, v_d)=d\geq d_{\widetilde{T}}(u_1, w)$,}\\[5pt]
&\text{$e_{\widetilde{T}}(w)=d_{\widetilde{T}}(v_0, w)=d_{\widetilde{T}}(u_1, w).$}
\end{align*}
By the definition of eccentricity matrix, we have $(\varepsilon(\widetilde{T}))_{u_1w}=d_{\widetilde{T}}(u_1, w)$.

It is easy to see that $d_T(u_1, w)<d_{\widetilde{T}}(u_1, w)$. Hence, for $w\in \bigcup_{\frac{d+1}{2} \leq j\leq d}V(T_j)$, we have $(\varepsilon(T))_{u_1w}<(\varepsilon(\widetilde{T}))_{u_1w}$.
\begin{Case}
$e_{T_i}(v_i)=i$.
\end{Case}
In this case, for $w\in \bigcup_{\frac{d+1}{2}\leq j\leq d}V(T_j)$, we have $d_{T}(u_1, w)=e_{T}(w)=e_{\widetilde{T}}(w)=d_{\widetilde{T}}(u_1, w)$. Hence, $(\varepsilon(T))_{u_1w}=(\varepsilon(\widetilde{T}))_{u_1w}$.

Together with Cases 1 and 2, we have $\varepsilon(T)\leqslant \varepsilon(\widetilde{T})$, with equality if and only if $d_{T_i}(u_1, v_i)=i$. By Lemmas \ref{lem2.2} and \ref{lem2.3}, we obtain $\varepsilon_1(T)\leqslant \varepsilon_1(\widetilde{T})$, with equality if and only if $d_{T_i}(u_1, v_i)=i$.

This completes the proof.
\end{proof}
\begin{figure}[!ht]
\begin{center}
\psfrag{a}{$v_0$}
\psfrag{b}{$v_1$}
\psfrag{c}{$v_i$}
\psfrag{d}{$v_{d-2}$}
\psfrag{e}{$v_{d-1}$}
\psfrag{f}{$v_d$}
\psfrag{g}{$u$}
\psfrag{o}{$\widetilde{T}$}
\psfrag{p}{$T$}
\psfrag{1}{$\cdots$}
\includegraphics[width=130mm]{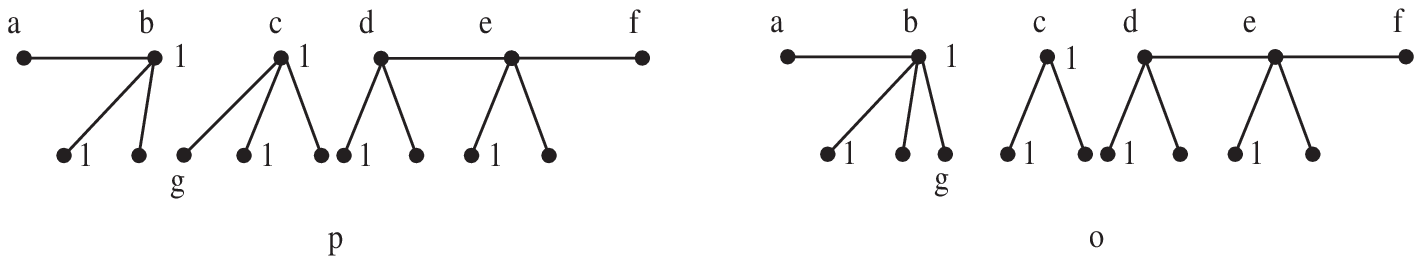} \\
\caption{Trees $T$ and $\widetilde{T}$ in Lemma \ref{lem3.4}.}
\end{center}
\end{figure}

\begin{lem}\label{lem3.4}
Given an $n$-vertex caterpillar tree $T$, $P_{d+1}=v_0v_1v_2 \cdots v_d$, ($d\geqslant 5$ is odd) is a diametrical path of $T$. Assume $d_{T}(v_i)\geq3$ ($2\leq i\leq\frac{d-1}{2}$). Moving a pendant edge, say $v_iu$ from $v_i$ to $v_1$ yields the tree $\widetilde{T}$ (see Fig. 3). Then $\varepsilon_1(T)<\varepsilon_1(\widetilde{T})$.
\end{lem}
\begin{proof}
It is obvious that $d_{T}(w,w')=d_{\widetilde{T}}(w,w')$ for $\{w,w'\}\subseteq V(T)\setminus \{u\}$ and $e_{T}(w)=e_{\widetilde{T}}(w)$ for $w\in V(T)\setminus \{u\}$. By the definition of eccentricity matrix, we have $(\varepsilon(T))_{ww'}=(\varepsilon(\widetilde{T}))_{ww'}$ for any $\{w,w'\}\subseteq V(T)\setminus \{u\}$.

If $w$ is a pendant neighbor of $v_{d-1}$ (may be $v_d$), then $d_{T}(u,w)=d-i+1=e_{T}(u)$ and $d_{\widetilde{T}}(u,w)=d=e_{\widetilde{T}}(u)$. Hence, $(\varepsilon(T))_{uw}=d-i+1<d=(\varepsilon(\widetilde{T}))_{uw}$. If $w$ is not a pendant neighbor of $v_{d-1}$, then we have $(\varepsilon(T))_{uw}<\min\{e_{T}(u),e_{T}(w)\}$. Thus, $(\varepsilon(T))_{uw}=0\leqslant(\varepsilon(\widetilde{T}))_{uw}$.

Clearly, $\varepsilon(T)\leqslant \varepsilon(\widetilde{T})$ and $\varepsilon(T)\neq\varepsilon(\widetilde{T})$. By Lemmas \ref{lem2.2} and \ref{lem2.3}, we obtain $\varepsilon_1(T)<\varepsilon_1(\widetilde{T})$.
\end{proof}

Our next main result in this section determines the unique tree among $\mathscr{T}_{n,d}$ with odd $d\geqslant 5$, having the maximum $\varepsilon$-spectral radius.
\begin{thm}\label{theorem2.4}
Let $T$ be in $\mathscr{T}_{n,d}$ with odd $d\geqslant 5$. Then
\begin{align*}
\rho\leq \max\Big\{\varepsilon_1(D_{n,d}^{0, n-d-1}), \varepsilon_1(D_{n,d}^{\left\lfloor\frac{n-d-1}{2}\right\rfloor, \left\lceil\frac{n-d-1}{2}\right\rceil})\Big\}
\end{align*}
with equality only if $T\cong D_{n,d}^{0, n-d-1}$ or $D_{n,d}^{\left\lfloor\frac{n-d-1}{2}\right\rfloor, \left\lceil\frac{n-d-1}{2}\right\rceil}$.
\end{thm}
\begin{proof}
Assume that $T$ is the tree with maximum $\varepsilon$-spectral radius among $\mathscr{T}_{n,d}$ with odd $d\geqslant 5$. By making frequent use of Lemmas \ref{lem3.3} and \ref{lem3.4}, we obtain that $T$ is isomorphic to some $D_{n,d}^{a,b}$, where $b\geq a\geq0$. Choose a diametrical path $P_{d+1}=v_0v_1v_2 \cdots v_d$ in $D_{n,d}^{a,b}$. Denote by $U=\{v_0, u_1, \cdots u_a\}$ the set of pendant neighbors of $v_1$ and let $W=\{v_d, w_1, \cdots w_b\}$ be the set of pendant neighbors of $v_{d-1}$ in $D_{n,d}^{a,b}$. By definition, the eccentricity matrix $\varepsilon(D_{n,d}^{a,b})$ is equal to

{\footnotesize
$$
\bordermatrix{
    & v_0 & v_1 & \cdots & v_{\frac{d-3}{2}} & v_{\frac{d-1}{2}} & v_{\frac{d+1}{2}} & v_{\frac{d+3}{2}} & \cdots & v_{d-1} & v_d & u_1 & \cdots & u_a & w_1 & \cdots & w_b \cr
v_0 & 0 & 0 & \cdots & 0 & 0 & \frac{d+1}{2} & \frac{d+3}{2} \ & \cdots & d-1 & d & 0 & \cdots & 0 & d & \cdots & d  \cr
v_1 & 0 & 0 & \cdots & 0 & 0 & 0 & 0 & \cdots & 0 & d-1 & 0 & \cdots & 0 & d-1 & \cdots & d-1  \cr
\vdots & \vdots & \vdots & \ddots &  \vdots & \vdots & \vdots & \vdots & \ddots & \vdots & \vdots & \vdots & \ddots & \vdots & \vdots & \ddots & \vdots  \cr
v_{\frac{d-3}{2}} & 0 & 0 & \cdots & 0 & 0 & 0 & 0 & \cdots & 0 & \frac{d+3}{2} & 0 & \cdots & 0 & \frac{d+3}{2} & \cdots & \frac{d+3}{2}  \cr
v_{\frac{d-1}{2}} & 0 & 0 & \cdots & 0 & 0 & 0 & 0 & \cdots & 0 & \frac{d+1}{2} & 0 & \cdots & 0 & \frac{d+1}{2} & \cdots & \frac{d+1}{2}  \cr
v_{\frac{d+1}{2}} & \frac{d+1}{2} & 0 & \cdots & 0 & 0 & 0 & 0 & \cdots & 0 & 0 & \frac{d+1}{2} & \cdots & \frac{d+1}{2} & 0 & \cdots & 0  \cr
v_{\frac{d+3}{2}} & \frac{d+3}{2} & 0 & \cdots & 0 & 0 & 0 & 0 & \cdots & 0 & 0 & \frac{d+3}{2} & \cdots & \frac{d+3}{2} & 0 & \cdots & 0  \cr
\vdots & \vdots & \vdots & \ddots & \vdots & \vdots & \vdots & \vdots & \ddots & \vdots & \vdots & \vdots & \ddots & \vdots & \vdots & \ddots & \vdots  \cr
v_{d-1} & d-1 & 0 & \cdots & 0 & 0 & 0 & 0 & \cdots & 0 & 0 & d-1 & \cdots & d-1 & 0 & \cdots & 0  \cr
v_d & d & d-1 & \cdots & \frac{d+3}{2} & \frac{d+1}{2} & 0 & 0 & \cdots & 0 & 0 & d & \cdots & d & 0 & \cdots & 0  \cr
u_1 & 0 & 0 & \cdots & 0 & 0 & \frac{d+1}{2} & \frac{d+3}{2} \ & \cdots & d-1 & d & 0 & \cdots & 0 & d & \cdots & d  \cr
\vdots & \vdots & \vdots & \ddots & \vdots & \vdots & \vdots & \vdots & \ddots & \vdots & \vdots & \vdots & \ddots & \vdots & \vdots & \ddots & \vdots  \cr
u_a & 0 & 0 & \cdots & 0 & 0 & \frac{d+1}{2} & \frac{d+3}{2} \ & \cdots & d-1 & d & 0 & \cdots & 0 & d & \cdots & d  \cr
w_1 & d & d-1 & \cdots & \frac{d+3}{2} & \frac{d+1}{2} & 0 & 0 & \cdots & 0 & 0 & d & \cdots & d & 0 & \cdots & 0  \cr
\vdots & \vdots & \vdots & \ddots & \vdots & \vdots & \vdots & \vdots & \ddots & \vdots & \vdots & \vdots & \ddots & \vdots & \vdots & \ddots & \vdots  \cr
w_b & d & d-1 & \cdots & \frac{d+3}{2} & \frac{d+1}{2} & 0 & 0 & \cdots & 0 & 0 & d & \cdots & d & 0 & \cdots & 0  \cr
}.
$$
}

Let ${\bf x}$ be a Perron eigenvector corresponding to $\rho:=\varepsilon_1(D_{n,d}^{a,b})$, whose coordinate with respect to vertex $v$ is ${\bf x}_{v}$. Since $\rho {\bf x}_u=\frac{d+1}{2}{\bf x}_{v_{\frac{d+1}{2}}}+\frac{d+3}{2}{\bf x}_{v_{\frac{d+3}{2}}}+\cdots+(d-1){\bf x}_{v_{d-1}}+d{\bf x}_{v_{d}}+d\sum_{i=1}^b{\bf x}_{w_{i}}$ for each $u\in U$, we can get ${\bf x}_u={\bf x}_{u'}$ for $\{u,u'\}\subseteq U$. Similarly, ${\bf x}_w={\bf x}_{w'}$ for $\{w,w'\}\subseteq W$. Then we obtain
\begin{align*}
\rho {\bf x}_u&=\frac{d+1}{2}{\bf x}_{v_{\frac{d+1}{2}}}+\frac{d+3}{2}{\bf x}_{v_{\frac{d+3}{2}}}+\cdots+(d-1){\bf x}_{v_{d-1}}+(b+1)d{\bf x}_{w};\\
\rho {\bf x}_w&=(a+1)d{\bf x}_{u}+(d-1){\bf x}_{v_{1}}+(d-2){\bf x}_{v_{2}}+\cdots+\frac{d+1}{2}{\bf x}_{v_{\frac{d-1}{2}}};\\
\rho {\bf x}_{v_1}&=(b+1)(d-1){\bf x}_w; \\
&\vdots\\
\rho {\bf x}_{v_{\frac{d-1}{2}}}&=(b+1)\cdot\frac{d+1}{2}{\bf x}_w; \\
\rho {\bf x}_{v_{d-1}}&=(a+1)(d-1){\bf x}_u; \\
&\vdots\\
\rho {\bf x}_{v_{\frac{d+1}{2}}}&=(a+1)\cdot\frac{d+1}{2}{\bf x}_u,
\end{align*}
for any $u\in U,w\in W$.

Let $(\frac{d+1}{2})^2+(\frac{d+3}{2})^2+\cdots+(d-1)^2=\frac{d(d-1)(7d-5)}{24}:=\Gamma(d)$. Hence,
\begin{align*}
\rho^2{\bf x}_u&=\rho\Big[\frac{d+1}{2}{\bf x}_{v_{\frac{d+1}{2}}}+\frac{d+3}{2}{\bf x}_{v_{\frac{d+3}{2}}}+\cdots+(d-1){\bf x}_{v_{d-1}}+(b+1)d{\bf x}_{w}\Big]\\
&=(a+1)(\frac{d+1}{2})^2{\bf x}_u+(a+1)(\frac{d+3}{2})^2{\bf x}_u+\cdots+(a+1)(d-1)^2{\bf x}_u+(b+1)d\rho{\bf x}_w\\
&=(a+1)\Gamma(d){\bf x}_u+(b+1)d\rho{\bf x}_w
\end{align*}
and
\begin{align*}
\rho^2{\bf x}_w&=\rho\Big[(a+1)d{\bf x}_{u}+(d-1){\bf x}_{v_{1}}+(d-2){\bf x}_{v_{2}}+\cdots+\frac{d+1}{2}{\bf x}_{v_{\frac{d-1}{2}}}\Big]\\
&=(a+1)d\rho{\bf x}_u+(b+1)(d-1)^2{\bf x}_w+(b+1)(d-2)^2{\bf x}_w+\cdots+(b+1)(\frac{d+1}{2})^2{\bf x}_w\\
&=(a+1)d\rho{\bf x}_u+(b+1)\Gamma(d){\bf x}_w.
\end{align*}
That is,
\begin{align*}
\rho^2{\bf x}_u-(a+1)\Gamma(d){\bf x}_u-(b+1)d\rho{\bf x}_{w}&=0,\\
-(a+1)d\rho{\bf x}_u+\rho^2{\bf x}_w-(b+1)\Gamma(d){\bf x}_w&=0.
\end{align*}

Since ${\bf x}_u\neq0$ and ${\bf x}_w\neq0$, $\rho$ is the largest root of
\begin{align}\label{3.1}
\left|\begin{array}{cc}
  t^2-(a+1)\Gamma(d)   & -(b+1)dt \\
  -(a+1)dt   & t^2-(b+1)\Gamma(d) \\
               \end{array}\right|=0. \notag
\end{align}
By calculation and the fact $a+b=n-d-1$, we have
\begin{align*}
0&=\Big[t^2-(a+1)\Gamma(d)\Big]\Big[t^2-(b+1)\Gamma(d)\Big]-(a+1)(b+1)d^2t^2\\
&=t^4-\Big[(a+1)\Gamma(d)+(b+1)\Gamma(d)+(a+1)(b+1)d^2\Big]t^2+(a+1)(b+1)(\Gamma(d))^2\\
&=t^4-\Big[\Gamma(d)(n-d+1)+(a+1)(b+1)d^2\Big]t^2+(a+1)(b+1)(\Gamma(d))^2.
\end{align*}
Let $\Delta(a,b):=\Big[\Gamma(d)(n-d+1)+(a+1)(b+1)d^2\Big]^2-4(a+1)(b+1)(\Gamma(d))^2$. By a direct calculation, $\rho^2$ is equal to
\begin{align}
\frac{1}{2}\Big[\Gamma(d)(n-d+1)+(a+1)(b+1)d^2\Big]+\frac{1}{2}\sqrt{\Delta(a,b)}
\end{align}
Let $x=(a+1)(b+1)$. We have $n-d\leq x\leq(\left\lfloor\frac{n-d-1}{2}\right\rfloor+1)(\left\lceil\frac{n-d-1}{2}\right\rceil+1)$, since $0\leq a\leq b\leq n-d-1$ and $a+b=n-d-1$. Hence,
\begin{align*}
\rho^2=\frac{1}{2}\Big[\Gamma(d)(n-d+1)+d^2x\Big]+\frac{1}{2}\sqrt{\Big[\Gamma(d)(n-d+1)+d^2x\Big]^2-4(\Gamma(d))^2x}.
\end{align*}

Clearly, for fixed $n$ and $d$, $f(x):=\frac{1}{2}\Big[\Gamma(d)(n-d+1)+d^2x\Big]+\frac{1}{2}\sqrt{\Big[\Gamma(d)(n-d+1)+d^2x\Big]^2-4(\Gamma(d))^2x}$ is a convex function in interval $[n-d, (\left\lfloor\frac{n-d-1}{2}\right\rfloor+1)\cdot(\left\lceil\frac{n-d-1}{2}\right\rceil+1)]$. The convexity of $f(x)$ implies that
\begin{align*}
f(x)\leq \max\Big\{f(n-d), f(\Big(\left\lfloor\frac{n-d-1}{2}\right\rfloor+1\Big)\cdot\Big(\left\lceil\frac{n-d-1}{2}\right\rceil+1\Big))\Big\}
\end{align*}
with equality only if $x=n-d$ or $(\left\lfloor\frac{n-d-1}{2}\right\rfloor+1)\cdot(\left\lceil\frac{n-d-1}{2}\right\rceil+1)$. Note that $\varepsilon_1(D_{n,d}^{0, n-d-1})=\sqrt{f(n-d)}$ and $\varepsilon_1(D_{n,d}^{\left\lfloor\frac{n-d-1}{2}\right\rfloor, \left\lceil\frac{n-d-1}{2}\right\rceil})=\sqrt{f(\Big(\left\lfloor\frac{n-d-1}{2}\right\rfloor+1\Big)\cdot\Big(\left\lceil\frac{n-d-1}{2}\right\rceil+1\Big))}$. Thus, we have
\begin{align*}
\rho\leq \max\Big\{\varepsilon_1(D_{n,d}^{0, n-d-1}), \varepsilon_1(D_{n,d}^{\left\lfloor\frac{n-d-1}{2}\right\rfloor, \left\lceil\frac{n-d-1}{2}\right\rceil})\Big\}
\end{align*}
with equality only if $T\cong D_{n,d}^{0, n-d-1}$ or $D_{n,d}^{\left\lfloor\frac{n-d-1}{2}\right\rfloor, \left\lceil\frac{n-d-1}{2}\right\rceil}$.
\end{proof}

Together with the proofs of Theorems \ref{thm3.2} and \ref{theorem2.4}, we obtain the following result.
\begin{thm}\label{theorem2.6}
Let $T$ be an $n$-vertex tree with odd diameter $d$. Then
\begin{align*}
\rho\leq \max\Big\{\varepsilon_1(D_{n,d}^{0, n-d-1}), \varepsilon_1(D_{n,d}^{\left\lfloor\frac{n-d-1}{2}\right\rfloor, \left\lceil\frac{n-d-1}{2}\right\rceil})\Big\}
\end{align*}
with equality only if $T\cong D_{n,d}^{0, n-d-1}$ or $D_{n,d}^{\left\lfloor\frac{n-d-1}{2}\right\rfloor, \left\lceil\frac{n-d-1}{2}\right\rceil}$. Especially, $D_{n,3}^{\lfloor \frac{n-4}{2}\rfloor, \lceil \frac{n-4}{2}\rceil}$ is the unique tree maximizing $\varepsilon$-spectral radius among all the trees in $\mathscr{T}_{n,3}$.
\end{thm}
\section{Trees with least $\varepsilon$-eigenvalues in $[-2-\sqrt{13},-2\sqrt{2})$}
In this section, we investigate trees with least $\varepsilon$-eigenvalues in $[-2-\sqrt{13},-2\sqrt{2})$. Let $\varepsilon_n(G)$ be the least $\varepsilon$-eigenvalue of a graph $G$ with order $n$.

For $p\geq0$ and $q\geq2$, let $H_{p,q}$ be the graph obtained from the star $S_{p+q+1}$ by attaching a pendant vertex to each of $q$ chosen pendant vertices (see Fig. 4).
\begin{figure}[!ht]
\begin{center}
\psfrag{a}{$a_1$}
\psfrag{b}{$a_p$}
\psfrag{c}{$w$}
\psfrag{d}{$b_1$}
\psfrag{e}{$b_2$}
\psfrag{f}{$b_q$}
\psfrag{g}{$c_1$}
\psfrag{h}{$c_2$}
\psfrag{i}{$c_q$}
\psfrag{1}{$\vdots$}
\includegraphics[width=50mm]{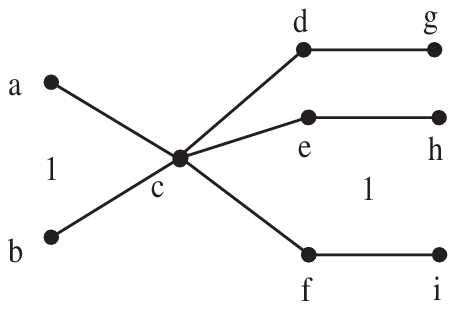} \\
\caption{$H_{p,q}$ with $p\geq0$ and $q\geq2$.}
\end{center}
\end{figure}
Let ${\bf J}_{n\times m}$ and ${\bf 0}_{n\times m}$ be respectively all-one and the all-zero $n\times m$ matrices. Let ${\bf J}_{n}={\bf J}_{n\times n}$, ${\bf 1}_{n}={\bf J}_{n\times 1}$, and ${\bf 0}_{n}={\bf 0}_{n\times 1}$.

The following is a key lemma that we will need in the proofs.
\begin{lem}[\cite{4}]\label{thm4}
Let ${\bf M}$, ${\bf N}$, ${\bf P}$ and ${\bf Q}$ be respectively $p\times p$, $p\times q$, $q\times p$ and $q\times q$ matrices, where ${\bf Q}$ is invertible. Then
\begin{equation*}
\left|
  \begin{array}{cc}
  {\bf M} & {\bf N} \\
  {\bf P} & {\bf Q}
  \end{array}
  \right|=|{\bf Q}|\cdot |{\bf M}-{\bf N}{\bf Q}^{-1}{\bf P}|.
\end{equation*}
\end{lem}
\begin{lem}\label{thm4.3}
For $p\geq0$ and $q\geq2$, the $\varepsilon$-polynomial of $H_{p,q}$ is
$$\lambda^{p+1}(\lambda^2+4\lambda-9)^{q-1}\big[\lambda^2+(4-4q)\lambda-(9pq+9q^2+9-14q)\big].$$
\end{lem}
\begin{proof}
Let $n=p+2q+1$. Let $w$ be the center of $S_{p+q+1}$, $A$ the set of pendant neighbors of $w$, $B$ the set of non-pendant neighbors of $w$, and $C$ the set of pendant vertices that are not neighbors of $w$. Then $\{w\}\bigcup A\bigcup B\bigcup C$ is a partition of $V(H_{p,q})$ (see Fig. 4), and with respect to this partition, we have
\begin{equation*}
\mathbf{\varepsilon}(H_{p,q})-\lambda\mathbf{I}_n
=\left(
  \begin{array}{cccc}
  -\lambda & \mathbf{0}_p^{\top} & \mathbf{0}_q^{\top} & 2\mathbf{1}_q^{\top}\\
  \mathbf{0}_p & -\lambda\mathbf{I}_{p} & \mathbf{0}_{p\times q} & 3\mathbf{J}_{p\times q}\\
  \mathbf{0}_q & \mathbf{0}_{q\times p} & -\lambda\mathbf{I}_{q} & 3\mathbf{J}_{q}-3\mathbf{I}_{q}\\
  2\mathbf{1}_q & 3\mathbf{J}_{q\times p}& 3\mathbf{J}_{q}-3\mathbf{I}_{q} & 4\mathbf{J}_{q}-(\lambda+4)\mathbf{I}_{q}
  \end{array}
  \right).
\end{equation*}

Let
\begin{equation*}
\mathbf{A}=\left(\begin{array}{cccccc}
                            1\ & 0\ & -2\ & -3\mathbf{1}_p^{\top}\ & (\frac{12}{\lambda+4}-3)\mathbf{1}_q^{\top} \\
                            0\ & 1\ & 0\ & \mathbf{0}_{p}^{\top}\ & (\frac{9}{\lambda+4})\mathbf{1}_q^{\top} \\
                            0\ & 0\ & -\lambda\ & \mathbf{0}_{p}^{\top}\ & (-\frac{6}{\lambda+4})\mathbf{1}_q^{\top} \\
                            \mathbf{0}_p\ & \mathbf{1}_p\ & \mathbf{0}_p\ & -\lambda\mathbf{I}_{p}\ & \mathbf{0}_{p\times q} \\
                            \mathbf{0}_q\ & \mathbf{1}_q\ & \mathbf{0}_q\ & \mathbf{0}_{qp}\ & (-\lambda+\frac{9}{\lambda+4})\mathbf{I}_q \\
                       \end{array}\right),
\mathbf{B}=\left(\begin{array}{ccccc}
                            -4\mathbf{1}^{\top}_q \\
                            -3\mathbf{1}^{\top}_q \\
                            2\mathbf{1}^{\top}_q \\
                            \mathbf{0}_{p\times q} \\
                            -3\mathbf{I}_q
                       \end{array}\right)
\end{equation*}
and
\begin{equation*}
\mathbf{C}=\left(\begin{array}{ccccc}
                            \mathbf{1}_q & \mathbf{0}_q & \mathbf{0}_q & \mathbf{0}_{q\times p} & \mathbf{0}_{q\times q}
                       \end{array}\right).
\end{equation*}
Then we have
\begin{align*}
&|\mathbf{A}-\mathbf{B}[-(\lambda+4)\mathbf{I}_q]^{-1}\mathbf{C}|\\
=&\left|\begin{array}{cccccc}
1-\frac{4q}{\lambda+4}\ & 0\ & -2\ & -3\mathbf{1}_p^{\top}\ & (\frac{12}{\lambda+4}-3)\mathbf{1}_q^{\top} \\
-\frac{3q}{\lambda+4}\ & 1\ & 0\ & \mathbf{0}_{p}^{\top}\ & (\frac{9}{\lambda+4})\mathbf{1}_q^{\top} \\
\frac{2q}{\lambda+4}\ & 0\ & -\lambda\ & \mathbf{0}_{p}^{\top}\ & (-\frac{6}{\lambda+4})\mathbf{1}_q^{\top} \\
\mathbf{0}_p\ & \mathbf{1}_p\ & \mathbf{0}_p\ & -\lambda\mathbf{I}_{p}\ & \mathbf{0}_{p\times q} \\
\frac{-3}{\lambda+4}\mathbf{1}_q & \mathbf{1}_q\ & \mathbf{0}_q\ & \mathbf{0}_{q\times p}\ & (-\lambda+\frac{9}{\lambda+4})\mathbf{I}_q \\
\end{array}\right|\\
=&\left|\begin{array}{cccccc}
1-\frac{4q}{\lambda+4}\ & -\frac{3p}{\lambda}\ & -2\ & -3\mathbf{1}_p^{\top}\ & (\frac{12}{\lambda+4}-3)\mathbf{1}_q^{\top} \\
-\frac{3q}{\lambda+4}\ & 1\ & 0\ & \mathbf{0}_{p}^{\top}\ & (\frac{9}{\lambda+4})\mathbf{1}_q^{\top} \\
\frac{2q}{\lambda+4}\ & 0\ & -\lambda\ & \mathbf{0}_{p}^{\top}\ & (-\frac{6}{\lambda+4})\mathbf{1}_q^{\top} \\
\mathbf{0}_p\ & \mathbf{0}_p\ & \mathbf{0}_p\ & -\lambda\mathbf{I}_{p}\ & \mathbf{0}_{p\times q} \\
\frac{-3}{\lambda+4}\mathbf{1}_q & \mathbf{1}_q\ & \mathbf{0}_q\ & \mathbf{0}_{q\times p}\ & (-\lambda+\frac{9}{\lambda+4})\mathbf{I}_q \\
\end{array}\right|\\
=&(-\lambda)^p\left|\begin{array}{cccccc}
1-\frac{4q}{\lambda+4}\ & -\frac{3p}{\lambda}\ & -2\ & (\frac{12}{\lambda+4}-3)\mathbf{1}_q^{\top} \\
-\frac{3q}{\lambda+4}\ & 1\ & 0\ & (\frac{9}{\lambda+4})\mathbf{1}_q^{\top} \\
\frac{2q}{\lambda+4}\ & 0\ & -\lambda\ & (-\frac{6}{\lambda+4})\mathbf{1}_q^{\top} \\
-\frac{3}{\lambda+4}\mathbf{1}_q & \mathbf{1}_q\ & \mathbf{0}_q\ & (-\lambda+\frac{9}{\lambda+4})\mathbf{I}_q \\
\end{array}\right|\\
=&(-\lambda)^p\left|\begin{array}{cccccc}
1-\frac{4q}{\lambda+4}-(\frac{12}{\lambda+4}-3)\frac{3q}{\lambda^2+4\lambda-9}\ & -\frac{3p}{\lambda}+(\frac{12}{\lambda+4}-3)(\frac{\lambda+4}{\lambda^2+4\lambda-9})q\ & -2\ & (\frac{12}{\lambda+4}-3)\mathbf{1}_q^{\top} \\
-\frac{3q}{\lambda+4}-\frac{9}{\lambda+4}\cdot\frac{3q}{\lambda^2+4\lambda-9}\ & 1+\frac{9q}{\lambda^2+4\lambda-9}\ & 0\ & (\frac{9}{\lambda+4})\mathbf{1}_q^{\top} \\
\frac{2q}{\lambda+4}+\frac{6}{\lambda+4}\cdot\frac{3q}{\lambda^2+4\lambda-9}\ & -\frac{6q}{\lambda^2+4\lambda-9}\ & -\lambda\ & (-\frac{6}{\lambda+4})\mathbf{1}_q^{\top} \\
\mathbf{0}_q & \mathbf{0}_q\ & \mathbf{0}_q\ & (-\lambda+\frac{9}{\lambda+4})\mathbf{1}_q \\
\end{array}\right|\\
=&(-\lambda)^p(\frac{\lambda^2+4\lambda-9}{\lambda+4})^q(-1)^q\left|\begin{array}{cccccc}
1-\frac{4q}{\lambda+4}+\frac{3\lambda}{\lambda+4}\cdot\frac{3q}{\lambda^2+4\lambda-9}\ & -\frac{3p}{\lambda}-\frac{3q\lambda}{\lambda^2+4\lambda-9}\ & -2 \\
-\frac{3q}{\lambda+4}-\frac{27q}{(\lambda+4)(\lambda^2+4\lambda-9)}\ & 1+\frac{9q}{\lambda^2+4\lambda-9}\ & 0 \\
\frac{2q}{\lambda+4}+\frac{18q}{(\lambda+4)(\lambda^2+4\lambda-9)}\ & -\frac{6q}{\lambda^2+4\lambda-9}\ & -\lambda \\
\end{array}\right|\\
=&(-1)^{p+q+1}\lambda^{p+1}(\lambda^2+4\lambda-9)^{q-1}(\lambda+4)^{-q}\cdot\big[\lambda^2+(4-4q)\lambda-(9pq+9q^2+9-14q)\big].
\end{align*}
Thus
\begin{align*}
&|\mathbf{\varepsilon}(H_{p,q})-\lambda\mathbf{I}_n|\\
=&\left|
  \begin{array}{cccccc}
  1 & 0 & -2 & -3\mathbf{1}_{p}^{\top} & -3\mathbf{1}_{q}^{\top} & -4\mathbf{1}_{q}^{\top}\\
  0 & 1 & 0 & \mathbf{0}_p^{\top} & \mathbf{0}_q^{\top} & -3\mathbf{1}_q^{\top}\\
  0 & 0 & -\lambda & \mathbf{0}_p^{\top} & \mathbf{0}_q^{\top} & 2\mathbf{1}_q^{\top}\\
  \mathbf{0}_p & \mathbf{0}_p & \mathbf{0}_p & -\lambda\mathbf{I}_{p} & \mathbf{0}_{p\times q} & 3\mathbf{J}_{p\times q}\\
  \mathbf{0}_q & \mathbf{0}_q & \mathbf{0}_q & \mathbf{0}_{q\times p} & -\lambda\mathbf{I}_{q} & 3\mathbf{J}_{q}-3\mathbf{I}_{q}\\
  \mathbf{0}_q & \mathbf{0}_q & 2\mathbf{1}_q & 3\mathbf{J}_{q\times p}& 3\mathbf{J}_{q}-3\mathbf{I}_{q} & 4\mathbf{J}_{q}-(\lambda+4)\mathbf{I}_{q}
  \end{array}
  \right|\\
=&\left|
  \begin{array}{cccccc}
  1 & 0 & -2 & -3\mathbf{1}_{p}^{\top} & -3\mathbf{1}_{q}^{\top} & -4\mathbf{1}_{q}^{\top}\\
  0 & 1 & 0 & \mathbf{0}_p^{\top} & \mathbf{0}_q^{\top} & -3\mathbf{1}_q^{\top}\\
  0 & 0 & -\lambda & \mathbf{0}_p^{\top} & \mathbf{0}_q^{\top} & 2\mathbf{1}_q^{\top}\\
  \mathbf{0}_p & \mathbf{1}_p & \mathbf{0}_p & -\lambda\mathbf{I}_{p} & \mathbf{0}_{p\times q} & \mathbf{0}_{p\times q}\\
  \mathbf{0}_q & \mathbf{1}_q & \mathbf{0}_q & \mathbf{0}_{q\times p} & -\lambda\mathbf{I}_{q} & -3\mathbf{I}_{q}\\
  \mathbf{1}_q & \mathbf{0}_q & \mathbf{0}_q & \mathbf{0}_{q\times p}& -3\mathbf{I}_{q} & -(\lambda+4)\mathbf{I}_{q}
  \end{array}
  \right|\\
=&\left|
  \begin{array}{cccccc}
  1 & 0 & -2 & -3\mathbf{1}_{p}^{\top} & (\frac{12}{\lambda+4}-3)\mathbf{1}_{q}^{\top} & -4\mathbf{1}_{q}^{\top}\\
  0 & 1 & 0 & \mathbf{0}_p^{\top} & \frac{9}{\lambda+4}\mathbf{1}_q^{\top} & -3\mathbf{1}_q^{\top}\\
  0 & 0 & -\lambda & \mathbf{0}_p^{\top} & -\frac{6}{\lambda+4}\mathbf{1}_q^{\top} & 2\mathbf{1}_q^{\top}\\
  \mathbf{0}_p & \mathbf{1}_p & \mathbf{0}_p & -\lambda\mathbf{I}_{p} & \mathbf{0}_{p\times q} & \mathbf{0}_{p\times q}\\
  \mathbf{0}_q & \mathbf{1}_q & \mathbf{0}_q & \mathbf{0}_{q\times p} & (-\lambda+\frac{9}{\lambda+4})\mathbf{I}_{q} & -3\mathbf{I}_{q}\\
  \mathbf{1}_q & \mathbf{0}_q & \mathbf{0}_q & \mathbf{0}_{q\times p}& \mathbf{0}_{q\times q} & -(\lambda+4)\mathbf{I}_{q}
  \end{array}
  \right|\\
=&\left|
  \begin{array}{cc}
  {\bf A} & {\bf B} \\
  {\bf C} & -(\lambda+4){\bf I}_q
  \end{array}
  \right|\\
=&|-(\lambda+4){\bf I}_q|\cdot |{\bf A}-{\bf B}[-(\lambda+4){\bf I}_q]^{-1}{\bf C}|\ \ \ \ \ \ \ \ \ \ \ \ \ \ \ \ \ \ \ \ \ \ \ \  \ \ \ \ \text{(by Lemma \ref{thm4})}\notag\\
=&(-1)^{p+1}\lambda^{p+1}(\lambda^2+4\lambda-9)^{q-1}\big[\lambda^2+(4-4q)\lambda-(9pq+9q^2+9-14q)\big].
\end{align*}

Note that the $\varepsilon$-polynomial of $H_{p,q}$ is $|\lambda\mathbf{I}_n-\varepsilon(H_{p,q})|=|-[\varepsilon(H_{p,q})-\lambda\mathbf{I}_n]|=(-1)^n|\varepsilon(H_{p,q})-\lambda\mathbf{I}_n|$. The result follows easily.
\end{proof}
\begin{lem}\label{lemma3.2}
Let $\varepsilon_n$ be the least $\varepsilon$-eigenvalue of $H_{p,q}$, where $n=p+2q+1$. Then we have $\varepsilon_n\leq-2- \sqrt{13}$, with equality if and only if one of the following conditions holds:
\begin{wst}
\item[{\rm (i)}] $p=0$ and $2\leq q\leq4$;
\item[{\rm (ii)}] $p=1$ and $q=2$ or $3$;
\item[{\rm (iii)}] $p=2$ and $q=2$.
\end{wst}
\end{lem}
\begin{proof}
In view of Lemma \ref{thm4.3}, it is easy to see that the distinct $\varepsilon$-eigenvalues of $H_{p,q}$ are $0, -2\pm \sqrt{13}$ and $\frac{1}{2}\Big[(4q-4)\pm\sqrt{(4-4q)^2+4(9pq+9q^2+9-14q)}\Big]$. It follows that $\varepsilon_n\leq-2- \sqrt{13}$. And we have
$$\varepsilon_n\in\Big\{-2- \sqrt{13}, \ \ \frac{1}{2}\Big[(4q-4)-\sqrt{(4-4q)^2+4(9pq+9q^2+9-14q)}\Big]\Big\}.$$

By simplifying the following inequality
$$\frac{1}{2}\Big[(4q-4)-\sqrt{(4-4q)^2+4(9pq+9q^2+9-14q)}\geq-2- \sqrt{13}.$$
We obtain
\begin{equation}
4q(9p+9q-4\sqrt{13}-22)\leq0.
\end{equation}
Since $q\geq2$, we have $9p+9q-4\sqrt{13}-22\leq0$ if and only if $\varepsilon_n=-2-\sqrt{13}$.
Note that $p\geq0$ and $q\geq2$, then $9p+9q-4\sqrt{13}-22\leq0$ implies $2\leq q\leq4$. We distinguish the following four cases.
\setcounter{Case}{0}
\begin{Case}
$p=0$.
\end{Case}
In this case, it is easy to see that $9p+9q-4\sqrt{13}-22\leq0$ for $2\leq q\leq4$.
\begin{Case}
$p=1$.
\end{Case}
In this case, we have $9p+9q-4\sqrt{13}-22\leq0$ if and only if $q=2$ or $3$.
\begin{Case}
$p=2$.
\end{Case}
In this case, it is easy to see that $9p+9q-4\sqrt{13}-22\leq0$ if and only if $q=2$.
\begin{Case}
$p\geq3$.
\end{Case}
In this case, by calculation, we have $9p+9q-4\sqrt{13}-22>0$ for $q\geq2$.

This completes the proof.
\end{proof}

\begin{thm}
Let $T$ be a tree with $n\geq3$ vertices. Then $\varepsilon_n(T)\in [-2-\sqrt{13}, -2\sqrt{2})$ if and only if one of the following conditions holds:
\begin{wst}
\item[{\rm (i)}] $T=P_4$;
\item[{\rm (ii)}] $T=D_{n,3}^{0,1}$;
\item[{\rm (iii)}] $T=H_{p,q}$ for $p=0$ and $2\leq q\leq4$;
\item[{\rm (iv)}] $T=H_{p,q}$ for $p=1$ and $q=2$ or $3$;
\item[{\rm (v)}] $T=H_{p,q}$ for $p=2$ and $q=2$.
\end{wst}
\end{thm}
\begin{proof}
Suppose $\varepsilon_n(T)\in [-2-\sqrt{13}, -2\sqrt{2})$. Let $d$ be the diameter of $T$. If $d=2$, then $T=S_n$, by Lemma \ref{lem2.5}, $\varepsilon_n(T)=-2$, a contradiction. If $d=5$, then $\varepsilon(P_6)$ is a principle submatrix of $\varepsilon(T)$. By a direct calculation and Lemma \ref{lem2.1}, we have $\varepsilon_n(T)\leq\varepsilon_6(P_6)\approx-8.0902<-2-\sqrt{13}$, a contradiction. If $d\geq6$, then by Lemma \ref{lem2.4}, we have $\varepsilon_n(T)\leq-6<-2-\sqrt{13}$, a contradiction. Thus $d=3$ or 4.

First suppose that $d=3$. Then $T$ is some $D_{n,3}^{a,b}$ with $b\geq a\geq0$. If $a=b=0$, then $T=D_{n,3}^{a,b}=P_4$. By calculating, we have $\varepsilon_4(P_4)=-4\in [-2-\sqrt{13}, -2\sqrt{2})$. This is (i). If $a=0, b=1$, then $T=D_{n,3}^{0,1}$. By calculating, we have $\varepsilon_5(D_{n,3}^{0,1})\approx-5.3752\in [-2-\sqrt{13}, -2\sqrt{2})$. This is (ii). Otherwise, $D_{n,3}^{1,1}$ or $D_{n,3}^{0,2}$ is an induced subgraph of $T$. Note that $\varepsilon(D_{n,3}^{1,1})$ or $\varepsilon(D_{n,3}^{0,2})$ is a principle submatrix of $\varepsilon(T)$. By a direct calculation, we obtain $\varepsilon_6(D_{n,3}^{1,1})\approx-7.1231$ and $\varepsilon_6(D_{n,3}^{0,2})\approx-6.4694$, and by Lemma \ref{lem2.1}, we have $\varepsilon_n(T)\leq\varepsilon_6(D_{n,3}^{1,1})<-2-\sqrt{13}$ or $\varepsilon_n(T)\leq\varepsilon_6(D_{n,3}^{0,2})<-2-\sqrt{13}$, a contradiction.

Next suppose that $d=4$. Then $D_{n,4}^{0,1}$ is an induced subgraph of $T$ or $T=H_{p,q}$ for some $p\geq0$ and $q\geq2$. In the former case, $\varepsilon(D_{n,4}^{0,1})$ is a principle submatrix of $\varepsilon(T)$. By a direct calculation and Lemma \ref{lem2.1}, we have $\varepsilon_n(T)\leq\varepsilon_6(D_{n,4}^{0,1})\approx-7.5621<-2-\sqrt{13}$, a contradiction. In the latter case, (iii), (iv) and (v) follow from Lemma \ref{lemma3.2}.

This completes the proof.
\end{proof}

\section{Concluding remarks}
{\noindent \bf Remark 1.}\ Theorem \ref{theorem2.6} characterizes trees with maximum $\varepsilon$-spectral radius among $n$-vertex trees with fixed odd diameter. For trees with maximum $\varepsilon$-spectral radius among $n$-vertex trees with fixed even diameter, it can not be determined similarly as Theorem \ref{theorem2.6} and an interesting research problem is put forward as follows.
\begin{pb}
Characterize the trees with maximum $\varepsilon$-spectral radius among $n$-vertex trees with fixed even diameter.
\end{pb}

\end{document}